\documentclass[10pt]{amsart}

\usepackage{thmtools}

\usepackage{stmaryrd}
\SetSymbolFont{stmry}{bold}{U}{stmry}{m}{n}
\usepackage[english]{babel}
\usepackage[utf8]{inputenc}
\usepackage[T1]{fontenc}   
\usepackage{enumitem}
\usepackage{mathrsfs}
\usepackage{bm}
\usepackage{stmaryrd}
\usepackage{amssymb,amsmath,amsthm}
\usepackage{xstring}
\usepackage{mathtools}
\usepackage{tikz-cd}
\usepackage{csquotes}
\usepackage{physics,derivative}
\usepackage{esint}

\usepackage[linktocpage]{hyperref}
\hypersetup{
	colorlinks=true,
	linkcolor=blue,
	filecolor=blue,
	citecolor = black,      
	urlcolor=cyan,
}

\newtheorem{theorem}{Theorem}[section]
\newtheorem{lemma}[theorem]{Lemma}
\newtheorem{proposition}[theorem]{Proposition}

\theoremstyle{definition}
\newtheorem{definition}[theorem]{Definition}

\theoremstyle{remark}
\newtheorem{remark}[theorem]{Remark}

\numberwithin{equation}{section}
\usepackage{amsrefs}

\usepackage{accents}
\DeclareMathOperator{\twodmass}{\Lambda}
\DeclareMathOperator{\ratioannu}{\beta}

\begin{document}

\title{A New Renormalized Volume Type Invariant}


\author{Wu Jinyang}
\address{225N MacLean Hall, Department of Mathematics, University of Iowa, Iowa City, IA, 52242}
\email{jwu124@uiowa.edu}
\thanks{The author acknowledges support from the Erwin and Peggy Kleinfeld Graduate Fellowship fund.}


\date{September 28, 2023}

\dedicatory{}

\commby{}

\begin{abstract}
	In this paper, we define a new conformal invariant on complete non-compact hyperbolic surfaces that can be conformally compactified to bounded domains in $\mathbb{C}$. We study and compute this invariant up to doubly-connected surfaces. Our results provide a new geometric criterion for choosing canonical representations of bounded domains in $\mathbb{C}$.
\end{abstract}

\maketitle

\section{Introduction}
Let $(M,g_M)$ be a complete non-compact hyperbolic surface of finite topological type that can be conformally compactified to a bounded domain in  $\mathbb{C}$. Consider all possible choices of $\Omega \subset \mathbb{C}$ such that $M$ can be conformally compactified into them and which are conformally equivalent to each other. In this article, we study a new canonical representation through proper geometric considerations. \par 
By the Riemann mapping theorem, all simply connected domains that are not $\mathbb{C}$ are conformally equivalent to a round disk, which can be viewed as a canonical model. It is well known that any doubly-connected region is conformally equivalent to a ring $\{z \in \mathbb{C}; \ratioannu < \abs{z} < 1\}$, where $0 \leq \ratioannu < 1$. For a reference, see \cite{MR510197}. Note that when $\ratioannu= 0$, the domain is simply a punctured disk. For multiply-connected domains, see \cite{MR510197} for some classification results. These results lead to a complete characterization of the moduli space of such surfaces. For example, all $n$-connected $(n>2)$ bounded domains with non-degenerate boundary components form a $3n-6$-dimensional space. 
In this article, we propose a new geometric point of view.\par 
 We set the notations for our construction. Let $M$ be a complete non-compact hyperbolic surface and $\Omega \subset \mathbb{C}$ be a fixed conformal compactification of $M$. Then there exists a smooth function $u$ on $\Omega$ such that $(M, e^{-2u} g_M)$ is isometric to $(\Omega, g_E)$. Let $v=e^{-u}$. Then $u$ and $v$ satisfy the following equations,
 \begin{equation}
\label{Liouvilleinftyboundaryproblem}
	\begin{cases}
		\Delta u = e^{2u} &\text{ in } \Omega,	\\
		u = + \infty & \text{ on }\partial \Omega,
	\end{cases}
 \end{equation}
and
    \begin{equation}\label{Liouvillevequation}
	\begin{cases}
		v \Delta v = \abs{\nabla v}^2- 1 & \text{ in }\Omega,\\
		v = 0 & \text{ on }\partial \Omega.
	\end{cases}
\end{equation}
Furthermore, $v$ has asymptotics near $C^{3,\alpha}$ boundary components of $\partial \Omega$ according to \cites{MR3758549,MR4308060},
\begin{equation}\label{boundaryexpansionv}
	v (z) = d(z) - \frac{1}{2} \kappa(y) d(z)^2 + c_3(y) d(z)^3 + O(d^{3 + \alpha}(z)),
\end{equation}
 $d(z)$ is the distance from $z$ to the boundary $\partial \Omega$, $y \in \partial \Omega$ is the point where $d(z) = \abs{z - y}$, and $\kappa(y)$ is the curvature at $y$. $c_3(y)$ is the first global term in \eqref{boundaryexpansionv}, which depends on the global geometry of $\Omega$. We remark that when $\partial \Omega$ is smooth, $v$ is the special boundary defining function as in \cite{MR1758076}*{Lemma 2.1}.\par 
We consider a scaling-free quantity following the work of \cite{MR4308060}. Assuming the outermost boundary component $\mathcal{C}$ of $\Omega$ is $C^{3,\alpha}$, we consider the following functional for $\mathcal{C}$,
\begin{equation}\label{InvariantLamba}
\lambda(\Omega,v)= -\int_{\mathcal{C}} \dd{l}\cdot \int_{\mathcal{C}} c_3(y) \dd{l(y)},
\end{equation}
where the line integral is taken under the Euclidean metric and $c_3$ is defined as in \eqref{boundaryexpansionv}. \par
The functional $\lambda$ was first studied by Shen and Wang \cite{MR4308060} in a slightly different setting. Using our notation, Shen and Wang have proven the following.
\begin{theorem}[{\cite{MR4308060}*{Theorem 1.1, 1.2, 1.3}}]\label{Shenwangresult}
	Let $\Omega$ be a bounded $C^{3,\alpha}$ domain, and $v$ be a solution to \eqref{Liouvillevequation} on $\Omega$. Let $\lambda$ be given as in \eqref{InvariantLamba}. Then 
	\begin{enumerate}[label=\alph*)]
		\item \label{positivemasstheoremSHEN} $\lambda (\Omega, v) \geq 0$, with equality if and only if $\Omega$ is a round disk;
		\item \label{gaptheoremSHEN} if $\Omega$ is multiply-connected, then
		\begin{equation}
			\lambda (\Omega, v)  > \frac{2\pi^2}{3} .
		\end{equation}
	\end{enumerate}
\end{theorem}
\autoref{Shenwangresult} can be regarded as a rigidity and gap theorem and may also be interpreted as a type of positive mass theorem. \par 

If $M$ is multiply-connected, each boundary component can be realized as the outermost boundary via a proper conformal transformation. Even with the fixed boundary component, \eqref{InvariantLamba} is not conformally invariant. We define the following renormalized volume-type invariant that is also conformally invariant for the pair $(M,g_M)$.

\begin{definition}\label{keydefinition}
Let $(M,g_M)$ be a complete, non-compact hyperbolic surface of finite topological type that can be conformally compactified to a bounded domain in $\mathbb{C}$. Let $\lambda$ be given as in \eqref{InvariantLamba}. Define
\begin{equation}\label{twodmassinf}
	\twodmass(M,g_M) = \inf_{\Omega}\left\{
\textstyle \lambda(\Omega,v)\left\vert
\begin{aligned}
	&(M,v^2g_M) \text{ is isometric to }(\Omega, g_E) \text{ and the}\\ &  \text{outermost boundary component of $\Omega$ is }C^{3,\alpha}.	\end{aligned}\right.
\right\}
\end{equation}
\end{definition}
By the Riemann mapping theorem, for any $M$, we may find a compactification such that the outermost boundary is smooth. Hence, $\Lambda$ is well-defined. 
\par 
Our main result is the following.
\begin{theorem}\label{Maintheorem}
Let $(M,g_M)$ be a complete non-compact hyperbolic doubly-connected surface with conformal compactification $(\Omega_0,g_E)$, where $\Omega_0 \subset\mathbb{R}^2$ is bounded. Let $\twodmass$ and $\lambda$ be given as in \eqref{twodmassinf} and \eqref{InvariantLamba} respectively. Then 
\begin{equation}
\twodmass(M,g_M) = \frac{2 \pi^2}{3} \biggl[\biggl(\frac{\pi}{\ln \ratioannu} \biggr)^2 + 1 \biggr],
\end{equation}
where $0 < \ratioannu < 1$ is the exponent of the modulus of continuity of $\Omega_0$, such that $\Omega_0$ is biholomorphic to $B_1 - \overline{B}_{\ratioannu}$. \par 
Moreover, for any bounded $\Omega$ biholomorphic to $\Omega_0$ with a $C^{3,\alpha}$ outermost boundary component, 
\begin{equation}\label{Maintheoremeq}
	\lambda(\Omega,v)  \geq \frac{2 \pi^2}{3} \biggl[\biggl(\frac{\pi}{\ln \ratioannu} \biggr)^2 + 1 \biggr],
\end{equation}
and equality holds if and only if $\Omega$ is the image of $B_1 - \overline{B}_{\ratioannu}$ under a composition of translation and homotheties. 
\end{theorem}

Our definition provides a new type of conformal invariant for two-dimensional hyperbolic conformally compact Einstein (CCE) manifolds. For higher dimensions, the renormalized volume is a preferred conformal invariant. However, it is topological in two dimensions for hyperbolic CCE manifolds. By \cite{MR2336463}*{Corollary 3.5} and \cite{MR1813434}*{Appendix A.1}, the renormalized volume of a conformally compact hyperbolic surface $(M,g_M)$ equals $-2\pi \chi(M)$. \autoref{keydefinition} provides an alternative invariant to work on. Our construction focuses on the geometric meaning of the next term in corresponding expansion \eqref{boundaryexpansionv}. We are hopeful that our invariant can serve as a substitute for the renormalized volume in two dimensions.

\autoref{Maintheorem} extends Shen and Wang's work \cite{MR4308060}, in which they compute $\twodmass$ when $M$ is simply connected, and separate this case with multiply-connected domains. Our result further characterizes doubly-connected $M$. \autoref{Maintheorem} shows $\twodmass$ can be attained for doubly-connected domains, specifically as $B_1-\overline{B}_{\ratioannu}$ up to a composition of translation and homotheties. In particular, when $\ratioannu=0$, \autoref{Maintheorem} also provides the sharp case of \ref{gaptheoremSHEN} in \autoref{Shenwangresult}.

An interesting byproduct of \autoref{Maintheorem} is a new geometric interpretation of the exponent of modulus of continuity $\ratioannu$. 

For CCE manifolds in dimension $4$, there is also some interest in studying specific compactifications; refer to \cites{MR1909634,MR3886176,MR4130465,MR4727591} for an incomplete list.

Our approach is partly inspired by Shen and Wang \cite{MR4308060}, though there are several key differences. Shen and Wang use conformal transformations and compare the boundary integral \eqref{InvariantLamba} of the underlying domain with that of the punctured disk. Subsequently, techniques involving the Schwarzian derivative are applied. In our work, we use more refined models. Specifically, we use Fourier series to derive inequalities and establish rigidity results. In particular, our method can recover their results. 

Note that by \cite{MR1237124}, solutions to \eqref{Liouvilleinftyboundaryproblem} on $C^{3,\alpha}$ domains exist and are unique. Thus, the functional $\lambda$ can be defined on bounded $C^{3,\alpha}$ domains $\Omega  \subset \mathbb{C}$ without specifying the solution to \eqref{Liouvillevequation} on $\Omega$. For punctured domains, we prove the following existence result and include the proof in the appendix.
\begin{proposition}\label{existenceextension}
Let $\Omega \subset \mathbb{C}$ be a smooth domain, then there exists a solution to \eqref{Liouvilleinftyboundaryproblem} on $\Omega-\{p_1,\cdots,p_n\}$ and it is unique among all solutions $u$ satisfying the following growth condition,
\begin{equation}\label{growthcondition}
	u(x) \geq - \log\bigl(-r_l \cdot \abs{x - p_l}\cdot \log \bigl(\frac{\abs{x-p_l}}{r_l}\bigr)\bigr),
\end{equation}
near $p_l$ for all $l=1,\cdots, k$, for $r_l > 0$ such that $B_{r_l}(p_l) \supset \Omega$.
\end{proposition}

We list some new directions. It is worthwhile to consider $\twodmass(M,g_M)$ for general multiply-connected domains. The resulting compactification will be a new canonical classification of bounded domains in $\mathcal{C}$, see \cite{MR510197} for some known classifications. We are seeking correct models for multiply-connected cases. For conformally compact hyperbolic surfaces of nonzero genus and conformally compact hyperbolic manifolds in general, defining $\twodmass(M,g_M)$ for such spaces would be interesting.

This paper is organized as follows. In \autoref{Techinicaltools} we prove a technical lemma. The proof of \autoref{Maintheorem} is split into two parts. In \autoref{oneconnecteddomains}, we prove the inequality in \eqref{Maintheoremeq}. In \autoref{rigidityofoneconnecteddomains}, we study when equality in \eqref{Maintheoremeq} holds.

The author would like to thank Biao Ma for introducing the work of Shen and Wang \cite{MR4308060}. The author would also like to thank Hao Fang, Lihe Wang and Biao Ma for helpful discussions, and Hao Fang for all the guidance on writing this paper. 

\section{A Technical Tool} \label{Techinicaltools}
In this section, we prove a technical lemma that provides an analytic result on annuli. The corresponding result for discs is easier. See \cite{MR924157} for more details. 

\begin{theorem}[Kellogg-Warschawski Theorem]\cite{MR1217706}*{Theorem 3.6}\label{keggtheorem}
Let $f$ map $B_1$ conformally onto a domain bounded by a $C^{k,\alpha}$ Jordan curve. Then $f$ has an extension in $C^{k,\alpha}(\overline{B}_1)$.
\end{theorem}

\autoref{keggtheorem} is a key ingredient in proving \autoref{technicalmellasquareroot}.

\begin{lemma}\label{technicalmellasquareroot}
Let $f \in B_{1} - \overline{B}_{\beta} \to \mathbb{C}$ be an orientation preserving biholomorphic map such that $f(B_1- \overline{B}_{\beta}) = \Omega - K$ for some compact subset $K$ of some bounded open $\Omega$. Then there exists a holomorphic $g \colon B_1 - \overline{B}_{\beta} \to \mathbb{C}$ such that $g^2= 1/\partial_zf$.
\end{lemma}
\begin{proof}
Fix $\beta < r < 1$. Let $\Omega_1$ be the bounded component of the complement of $f(\partial B_r)$, and let $f_1 \colon \Omega_1 \to B_r$ be an orientation preserving biholomorphic map, whose existence is guaranteed by the Riemann mapping theorem. By \autoref{keggtheorem}, $f_1$ extends to a smooth diffeomorphism between $\overline{\Omega}_1$ and $\overline{B}_r$. Consider $h = f_1 \circ f$, which is an orientation preserving diffeomorphism from $\overline{B}_r - \overline{B}_\beta$ to $\overline{B}_r - f_1(K)$.  \par 
 
Since $h$ restricts to a self-homeomorphism of $\partial B_r$,
\begin{equation}\label{zeroimaginarypart}
0 = \frac{\partial}{\partial t}\bigl[(h\overline{h})(e^{it})\bigr] = 2 \real\bigl(i\partial_zh(e^{it})\cdot e^{it}\cdot  \overline{h}(e^{it})\bigr) = 2 \imaginary\bigl( \frac{ \partial_zh(e^{it})\cdot  e^{it}}{h(e^{it})}\bigr).
\end{equation}
By \eqref{zeroimaginarypart}, there is $c\colon \mathbb{R} \to \mathbb{R}$ such that
\begin{equation}\label{technicallemmascalign1}
\frac{\partial_zh(e^{it})\cdot  e^{it}}{h(e^{it})} = c(t) > 0,
\end{equation}
where the inequality follows from $h$ being orientation preserving. Let $\gamma \colon [0, 2\pi] \to \mathbb{C}$ be a closed path defined by $\gamma(t) = e^{it}$. It follows that 
\begin{equation}\label{Indexcalculation}
\begin{split}
\text{Ind}_{\partial_z h( \gamma)}(0) &= \text{Ind}_{h(\gamma)/\gamma}(0)\\
&= \frac{1}{2 \pi i}\int_0^{2 \pi} \frac{\partial_z h(\gamma)\cdot \gamma_t/\gamma - h(\gamma)/\gamma^2 \cdot \gamma_t}{h(\gamma)/\gamma}\dd{t}\\
&= \frac{1}{2 \pi i}\int_0^{2 \pi}  \frac{\partial_z h(\gamma)\cdot \gamma_t}{h(\gamma)}- \frac{\gamma_t}{\gamma}\dd{t}\\
&= \text{Ind}_{h(\gamma)}(0)- \text{Ind}_{\gamma}(0) = 0
\end{split}
\end{equation}
where $\text{Ind}$ is the index of a point with respect to a closed path and the first equality follows from \eqref{technicallemmascalign1}.\par 
 By \cite{MR924157}*{Theorem 13.11}, \eqref{Indexcalculation} shows that $\log(\partial_z h)$ is well defined on $B_{r} - \overline{B}_\beta$. Since $f_1$ is defined on $\Omega_1$, which is a simply connected domain, \cite{MR924157}*{Theorem 13.11} implies that $\log \partial_zf_1$ is well defined on $\Omega_1$. Hence
 \begin{equation}
 	\partial_zf(z)= \frac{\partial_z h(z)}{\partial_zf_1(f(z))}=  \exp \bigl(\log \partial_z h(z) - \log \partial_z f_1 (f(z))\bigr),\label{Indexcalculation1}
 \end{equation}
for $z \in B_r - \overline{B}_\beta$. 
Invoking \cite{MR924157}*{Theorem 13.11} once again, \eqref{Indexcalculation1} implies that $\text{Ind}(\partial_z f(\gamma))(0) = 0$, therefore $\log \partial_z f$ is well defined on $B_1 - \overline{B}_\beta$. In particular, there exists a holomorphic $g \colon B_1 - \overline{B}_\beta \to \mathbb{C}$ such that $g^2 = 1/\partial_z f$.
\end{proof}

\section{doubly-connected Domains}
\label{oneconnecteddomains}
In this section, we prove our main theorem. In the simply connected case, it has been shown by \cite{MR4308060}.\par  
Let $u $ be a solution to \eqref{Liouvilleinftyboundaryproblem} on $\Omega$. Let $f \colon B_{1}- \overline{B}_{\ratioannu} \to \Omega$ be a biholomorphic function, and $u_{\ratioannu}$ be the solution to \eqref{Liouvilleinftyboundaryproblem} for $\Omega_{\ratioannu} = B_1 - \overline{B}_{\ratioannu}$, then $(\Omega, (f^{-1})^\ast e^{2u_{\ratioannu}} g_E)$ is complete with Gaussian curvature $-1$, hence
\begin{equation}\label{uconformalrelation}
	u = (u_{\ratioannu} - \ln \abs{f_z}) \circ f^{-1}
\end{equation}
is a solution to \eqref{Liouvilleinftyboundaryproblem} in $\Omega$. For $v = e^{-u}$ and $v_\beta = e^{-u_\beta}$, this relation becomes  
\begin{equation}\label{vconformalrelation}
	v = (v_{\ratioannu} \cdot \abs{f_z}) \circ (f^{-1}).
\end{equation}\par 
\begin{remark}
We quote the following explicit solution from \cite{MR4308060}.
The solution to \eqref{Liouvillevequation} on $B_1 - \overline{B}_{\ratioannu}$ is explicitly written as follows,
\begin{equation}
	v_{\ratioannu}(r) = -r \cdot \biggl(\frac{1}{\pi}\ln \ratioannu\biggr)\cdot \sin \frac{\ln r}{1/\pi \cdot \ln \ratioannu}.
\end{equation}
At $\partial B_1$, it has expansion
\begin{equation}\label{annulusc3expansion}
	v_{\ratioannu} = d - \frac{1}{2} d^2 - \frac{1}{6} \biggl[\biggl( \frac{\pi}{\ln\ratioannu}\biggr)^2 + 1 \biggr]d^3  + O(d^4),
\end{equation}
where $d(x) = 1-\lvert x \rvert$, and the first global term of $v_{\ratioannu}$ near $\partial B_1$ is 
\begin{equation}\label{firstglobalterm3beta3}
c_{\ratioannu,3} =  - \frac{1}{6} \biggl[\biggl( \frac{\pi}{\ln\ratioannu}\biggr)^2 + 1 \biggr].
\end{equation}
\end{remark}
The following lemma is proved in \cite{MR4308060}, here we include it for completeness.
\begin{lemma}\label{c3formula}
	Let $\Omega$ be a bounded domain with $C^{3,\alpha}$ outermost boundary component $\mathcal{C}$.
	Let $f \colon B_1 - \overline{B}_{\ratioannu} \to \Omega$ be a biholomorphic map. Let $c_3, c_{\ratioannu,3}$ be the first global terms of the solutions $v,v_\beta$ to \eqref{Liouvillevequation} on $\Omega, B_1- \overline{B}_{\ratioannu}$ near $\mathcal{C}, \partial B_1$ respectively. If $f$ takes $\partial B_1$ to $\mathcal{C}$, then  
\begin{equation}
\int_{\mathcal{C}} -6 c_3 \dd{l} = \int_{\partial B_{1}}  \frac{-6 c_{\ratioannu,3} }{\abs{f_z}} + 2 \pi\fint_{\partial B_{1}} \pdv[order = 2]{}{r} \frac{1}{\abs{f_z}} - 2 \pi 	\fint_{\partial B_{1}} \pdv{}{r} \frac{1}{\abs{f_z}}.
\end{equation}
\end{lemma}
\begin{proof}
By \cite{MR4308060}, the $c_3$ term in the expansion of $v$ can be expressed, in a way independent of the distance function $d(x)$, as
\begin{equation}
	-6 c_3 = \partial_N \Delta v,
\end{equation}
where $N$ is the outward pointing normal and $\Delta$ is the Laplacian under metric $(\Omega, g_E)$. \par 
Let $w$ and $z$ be the standard euclidean coordinate on $\Omega$ and $B_1 - \overline{B}_{\ratioannu}$ respectively. Then 
\begin{equation}
\label{c3expressionintermedaite}
\begin{split}
	\int_{\mathcal{C}} -6 c_3 \dd{l}
	&=\int_{f(\partial B_1)} \partial_N \Delta_w v \dd{l} = \int_{\partial B_1} f^\ast (\partial_N \Delta_w v \cdot \dd{l})\\
	&= \int_{\partial B_1} ((f^{-1})_\ast \partial_N) (\Delta_w v  \circ f) f^\ast \dd{l} \\
	&= \int_{\partial B_1} \frac{1}{\abs{f_z}} \pdv{}{r} \bigl[\Delta_{\abs{f_z}^2 g_E} (v_\beta \cdot \abs{f_z})\bigr] \abs{f_z} \dd{l}\\
		&= \int_{\partial B_1}  \pdv{}{r} \biggl[\frac{1}{\abs{f_z}^2}\Delta_{g_E}(v_\beta \cdot \abs{f_z})\biggr] \dd{l}.\\
	\end{split}
\end{equation}
We can compute the terms involving with $v_\beta$  on $\partial B_1$ explicitly: \begin{equation}
	\partial_r v_{\ratioannu}  = -1, \qquad \Delta_z v_{\ratioannu}  = -2, \qquad \partial_r \Delta_z v_{\ratioannu}  = -6 c_{\ratioannu,3} = \biggl[\biggl( \frac{\pi}{\ln \ratioannu}\biggr)^2 + 1 \biggr].
\end{equation}
Thus
\begin{equation}\label{c3expressioninterbf}
	\begin{split}
		\int_{\mathcal{C}} -6 c_3 \dd{l}
		&= \int_{\partial B_{1}} \biggl(\pdv{}{r} \frac{1}{\abs{f_z}} \biggr) \Delta v_{\ratioannu} + \int_{\partial B_{1}} \frac{1}{\abs{f_z}} \pdv{}{r}\Delta v_{\ratioannu} - 2 \int_{\partial B_{1}} \pdv[order = 2]{}{r} v_{\ratioannu} \pdv{}{r} \frac{1}{\abs{f_z}}\\
		& \quad  -2 \int_{\partial B_{1}} \pdv{}{r} v_{\ratioannu} \pdv[order = 2]{}{r} \frac{1}{\abs{f_z}} + \int_{\partial B_{1}} \pdv{}{r} v_{\ratioannu} \Delta \frac{1}{\abs{f_z}} + \int_{\partial B_{1}} v_{\ratioannu} \pdv{}{r} \Delta \frac{1}{\abs{f_z}}\\
		&=-6 c_{\ratioannu,3} \int_{\partial B_{1}}  \frac{1}{\abs{f_z}} +2 \int_{\partial B_{1}} \pdv[order = 2]{}{r} \frac{1}{\abs{f_z}} - \int_{\partial B_{1}} \Delta \frac{1}{\abs{f_z}}\\
		&= \int_{\partial B_{1}}  \frac{-6 c_{\ratioannu,3} }{\abs{f_z}} + 2 \pi\fint_{\partial B_{1}} \pdv[order = 2]{}{r} \frac{1}{\abs{f_z}} - 2 \pi 	\fint_{\partial B_{1}} \pdv{}{r} \frac{1}{\abs{f_z}}. 
	\end{split}
\end{equation}
In the second equality in \eqref{c3expressioninterbf}, we used the fact that $f$ extends to a $C^{3,\alpha}$ diffeomorphism from $\overline{B}_1 - B_\beta$ onto its image. The regularity of the extension is an easy application of \autoref{keggtheorem}, which indicates the existence of a $C^{3,\alpha}$ diffeomorphism $f_1$ from the closure of the region bounded by the $C^{3,\alpha}$ curve $\mathcal{C}$ to $\overline{B}_1$. The composition $f_1 \circ f$ is then a map from $B_1 - \overline{B}_{\beta}$ to $B_1 - K$ for some compact $K \subset B_1$ that sends $\partial B_1$ to $\partial B_1$. An application of the Schwarz reflection principle \cite{MR924157}*{Theorem 11.14} extends $f_1 \circ f$ to a $C^{3,\alpha}$ function on $\overline{B}_1 - \overline{B}_{\beta}$. Composing $(f_1)^{-1}$ with this extension of $f_1 \circ f$ gives us the desired $C^{3,\alpha}$ extension of $f$.
\end{proof}

The following theorem is an improvement of \ref{gaptheoremSHEN} in \autoref{Shenwangresult}.
\begin{theorem}\label{Mainresultannulus1}
	Let $\Omega$ be a doubly-connected domain with the outermost boundary component $\mathcal{C}$ being $C^{3,\alpha}$. If $\Omega$ is biholomorphic to $B_1 - \overline{B}_{\ratioannu}$, then 
	\begin{equation}\label{Mainresultinequality}
		\lambda(\Omega,v) \geq \frac{2 \pi^2}{3}\biggl[\biggl( \frac{\pi}{\ln\ratioannu}\biggr)^2 + 1 \biggr].
	\end{equation}
\end{theorem}
\begin{proof}
Let $f \colon B_1- \overline{B}_{\ratioannu} \to \Omega$ be an orientation preserving biholomorphic map that takes $\partial B_1$ to $\mathcal{C}$. 
By \autoref{c3formula},
\begin{equation}\label{Mainresultannulus1pat0}
	\int_{\mathcal{C}} -6 c_3 \dd{l} = \int_{\partial B_{1}}  \frac{-6 c_{\ratioannu,3} }{\abs{f_z}} + 2 \pi\fint_{\partial B_{1}} \pdv[order = 2]{}{r} \frac{1}{\abs{f_z}} - 2 \pi 	\fint_{\partial B_{1}} \pdv{}{r} \frac{1}{\abs{f_z}}.
\end{equation}
 By \autoref{technicalmellasquareroot}, there is holomorphic $g \colon B_1 - \overline{B}_{\beta} \to \mathbb{C}$ such that $g^2=1/f_z$. \par 

 Consider the Laurent expansion of $g$, 
\begin{equation}
g = \sum_{k= -\infty}^\infty b_k z^k,
\end{equation} which leads to,  
\begin{equation}
	\fint_{\partial B_r} \abs{g^2} = 
\frac{1}{2 \pi r}\int_{\partial B_r} \abs{g^2}= \frac{1}{2 \pi r}\int_{\partial B_r} g \cdot \overline{g} = \sum_{k = -\infty}^\infty \abs{b_k}^2 r^{2k}.
\end{equation}
Therefore,
\begin{equation}
	\begin{split}\label{Mainresultkeyinequality}
r^2\fint_{\partial B_{r}} \pdv[order = 2]{}{r} \abs{g^2} - r\fint_{\partial B_{r}} \pdv{}{r} \abs{g^2} &= \biggl(r^2\pdv[order = 2]{}{r} - r\pdv{}{r} \biggr)\sum_{k = -\infty}^\infty \abs{b_k}^2 r^{2k}\\
&= \sum_{k = -\infty}^\infty \abs{b_k}^2 2k(2k - 2) r^{2k} \geq 0.
	\end{split}
\end{equation}\par  
Let $r \to 1$, we conclude from \eqref{Mainresultannulus1pat0} and \eqref{Mainresultkeyinequality} that 
\begin{equation}
	\begin{split}
	\int_{\mathcal{C}} -6 c_3 \dd{l} &= \int_{\partial B_{1}}  \frac{-6 c_{\ratioannu,3} }{\abs{f_z}} + 2 \pi\fint_{\partial B_{1}} \pdv[order = 2]{}{r} \frac{1}{\abs{f_z}} - 2 \pi 	\fint_{\partial B_{1}} \pdv{}{r} \frac{1}{\abs{f_z}}\\ &\geq \int_{\partial B_{1}}  \frac{-6 c_{\ratioannu,3} }{\abs{f_z}}.
	\end{split}
\end{equation}
Finally, by H\"{o}lder's inequality and \eqref{firstglobalterm3beta3},
\begin{equation}\label{c3expressionfinal}
	\begin{split}
		\lambda(\Omega,v) &\geq \frac{1}{6}\biggl[\biggl( \frac{\pi}{\ln\ratioannu}\biggr)^2 + 1 \biggr] \cdot \int_{\partial B_1} 
		\frac{1}{\abs{f_z}}\dd{l} \cdot \int_{\partial B_1} \abs{f_z}\dd{l}\\
		&  \geq \frac{1}{6}\biggl[\biggl( \frac{\pi}{\ln\ratioannu}\biggr)^2 + 1 \biggr] \biggl(\int_{\partial B_1} \frac{1}{\abs{f_z}^{1/2}} \cdot  \abs{f_z}^{1/2} \biggr)^2\\
		& = \frac{2 \pi^2}{3}\biggl[\biggl( \frac{\pi}{\ln\ratioannu}\biggr)^2 + 1 \biggr], 
	\end{split}
\end{equation}
completing the proof.
\end{proof}
The first part of \ref{positivemasstheoremSHEN} in \autoref{Shenwangresult} can be recovered by the same proof as in \autoref{Mainresultannulus1} with little modification. First, note that for a simply connected domain $\Omega$, and biholomorphic $f \colon B_1 \to \Omega$. \autoref{c3formula} holds after we change $c_{\ratioannu,3}$ to the first global term $c_{B_1, 3}$ of the solution $v_{B_1}$ to \eqref{Liouvillevequation} on $B_1$ near $\partial B_1$. $v_{B_1}$ takes the following form,
\begin{align}
v_{B_1}&= \frac{1-r}{2}, &  c_{B_1,3}&=0.
\end{align}
Then 
\begin{equation}
\int_{\mathcal{C}} -6 c_3 \dd{l} = 2 \pi\fint_{\partial B_{1}} \pdv[order = 2]{}{r} \frac{1}{\abs{f_z}} - 2 \pi 	\fint_{\partial B_{1}} \pdv{}{r} \frac{1}{\abs{f_z}},
\end{equation}
which is $\geq 0$ by \eqref{Mainresultkeyinequality}. Thus, we conclude that for any simply connected domain $\Omega$,
\begin{equation}
\lambda(\Omega,v) \geq 0.
\end{equation}
The other part of \ref{positivemasstheoremSHEN} in \autoref{Shenwangresult} will be recovered in the following section.
\section{Rigidity of doubly-connected Domains}
\label{rigidityofoneconnecteddomains}
In this section, we prove a rigidity theorem for doubly-connected domains.\par  
Let $f \colon B_1 - \overline{B}_{\ratioannu} \to \Omega$ be an orientation preserving biholomorphic map. Consider 
\begin{align}\label{annuluacaseAC3study}
	A(t) &= \frac{1}{2 \pi  e^t}\cdot \int_{\partial B_{ e^t}(0)}, \frac{1}{\abs{f_z}} & B(t) &= A_{tt}(t) - 2A_t(t),
\end{align}
for $t \in (\ln \ratioannu, 0)$.
\begin{lemma}
Let $r = e^t$, then
\begin{equation}\label{ABequivalence}
B(t) = r^2 \fint_{\partial B_r(0)} \pdv[order = 2]{}{r}\frac{1}{\abs{f_z}} - r \fint_{\partial B_r(0)}\pdv{}{r}\frac{1}{\abs{f_z}} 
\end{equation}
for $t\in (\ln \ratioannu,0)$.
\end{lemma}
\begin{proof}
For convenience, take 
		\begin{equation}
\phi(r) = \frac{1}{2 \pi r} \int_{\partial B_r(0)} \frac{1}{\abs{f_z}},
		\end{equation}
	so that $\phi(r) = A(t)$. For fixed $r \in (\ratioannu, 1)$, we have 
	\begin{equation}
			\pdv{}{t} A = e^t \cdot \biggl(\pdv{}{r} \phi \biggr)(r)= \frac{r}{2 \pi r} \int_{\partial B_r} \pdv{}{r} \frac{1}{\abs{f_z}} \dd{l},
	\end{equation}
	and 
	\begin{equation}
		\begin{split}
			\pdv[order = 2]{}{t} A &= e^t \cdot \biggl(\pdv{}{r} \phi \biggr)(r) + e^{2t}\biggl(\pdv[order = 2]{}{r} \phi \biggr)(r)\\
			&= \frac{r}{2 \pi r} \int_{\partial B_r} \pdv{}{r} \frac{1}{\abs{f_z}} \dd{l}+ \frac{r^2}{2 \pi r}\int_{\partial B_r} \pdv[order = 2]{}{r} \frac{1}{\abs{f_z}} \dd{l}.
		\end{split}
	\end{equation}	
Thus the claim follows.
\end{proof}
By \autoref{c3formula} and \eqref{ABequivalence}, we see that  
\begin{equation}
\lim_{t \to 0}	2\pi \cdot B(t) =
\int_{\mathcal{C}} -6 c_3 - \int_{\partial B_1} \frac{-6 c_{\ratioannu,3}}{\abs{f_z}},
\end{equation}
which in a rough sense, measures the difference between $\lambda(\Omega,v)$ and $\lambda(B_1 - \overline{B}_{\ratioannu},v_{\ratioannu})$.
 \autoref{keyc3inequalityannulus} gives a description of the corresponding holomorphic functions $f$ in the case when $B(0)$ vanishes.

\begin{proposition}\label{keyc3inequalityannulus}
Let $A(t)$ and $B(t)$ be defined by \eqref{annuluacaseAC3study}. Then
	\begin{equation}
		B(t) \geq 0,
	\end{equation}
	for all $t \in (\ln \ratioannu, 0)$. \par If for some $t_0\in (\ln \ratioannu,0)$, $B(t_0) = 0$ or $\lim_{t\to 0}B(t) = 0$, then 
	\begin{equation}
	B(t) = 0,
	\end{equation}
for all $t\in (\ln\ratioannu,0)$ and 
	\begin{equation}
		f(z) = C_1 + C_2 z	  \quad \text{or} \quad f(z) = C_1 + \frac{C_2}{z + C_3}.
	\end{equation}
\end{proposition}

\begin{proof}
The non-negativeness of $B$ follows from \eqref{Mainresultkeyinequality} and \eqref{ABequivalence}.
	
Let $f \colon B_1 - \overline{B}_{\ratioannu} \to \Omega$ be a biholomorphic orientation preserving function. By \autoref{technicalmellasquareroot}, there is holomorphic $g \colon B_1 - \overline{B}_\beta \to \mathbb{C}$ such that $g^2=1/f_z$. Consider Laurent expansions of $g$,
	\begin{equation}
		g = \sum_{k = -\infty}^\infty b_k z^k.
\end{equation}

By \eqref{Mainresultkeyinequality} and \eqref{ABequivalence}, 
\begin{equation}\label{Mainresultkeyinequalitynext}
	B(t) = \sum_{k = -\infty}^\infty \abs{b_k}^2 2k(2k - 2) r^{2k}.
\end{equation} 
In particular, for any $k \neq 0,1$,
\begin{equation}
\abs{b_k}^2 \leq r^{-2k}\cdot B(t)
\end{equation}
for all $t \in (\ln \ratioannu,0)$, $r = e^t$. By the assumption of \autoref{keyc3inequalityannulus}, either $B(t) =0$ for some $t$ or $\lim_{t \to 0} B(t) = 0$, it follows that,
\begin{equation}\label{zeroboundforceconverterm}
\abs{b_k}^2	 = 0.  
\end{equation}
for $k \neq 0,1$.\par 
	Then $f_z$ is of the form 
	\begin{equation}
		f_z = \frac{1}{(b_0 + b_1 z)^2}.
	\end{equation} 
	Suggesting that 
	\begin{equation}\label{maintheorem2pt1}
		f = \begin{cases}
			C + b_0^{-2} z 	& \text{ if }b_1 = 0\\
			C - b_1^{-2}\cdot \bigl(b_0/b_1 + z\bigr)^{-1} & \text{ if }b_1 \neq 0,
		\end{cases}
	\end{equation}
	for some constant $C \in \mathbb{C}$.\par 
	It is easy to check that when \eqref{maintheorem2pt1} holds, $B(t) = 0$ for all $t \in (\ln r, 0)$.
\end{proof}

\begin{theorem}\label{Mainresultannulus}
	Let $\Omega$ be a doubly-connected domain with $C^{3, \alpha}$ outermost boundary. If $\Omega$ is biholomorphic to $B_1 - \overline{B}_{\ratioannu}$, and 
	\begin{equation}\label{Mainresultequality}
		\lambda(\Omega,v) = \frac{2 \pi^2}{3}\biggl[\biggl( \frac{\pi}{\ln\ratioannu}\biggr)^2 + 1 \biggr],
	\end{equation}
 	then $\Omega$ can be obtained from $B_1 - \overline{B}_{\ratioannu}$ by a composition of translation and homotheties. 
\end{theorem}
\begin{proof}
	By the proof of \autoref{Mainresultannulus1} and \eqref{ABequivalence}, the equality \eqref{Mainresultequality} holds when $\lim_{t \to 0} B(t) = 0$ and the H\"{o}lder's inequality in \eqref{c3expressionfinal} is an equality. By \autoref{keyc3inequalityannulus}, $\lim_{t \to 0} B(t) = 0$ forces the biholomorphic map $f \colon B_1 - \overline{B}_{\ratioannu} \to \Omega$ to be a M\"obius transformation, hence $f_z$ is of the form,  \begin{equation}\label{lastbitofit}
		f_z = \frac{1}{(b_0+b_1z)^2}.
	\end{equation} The H\"{o}lder's inequality in \eqref{c3expressionfinal} becomes an equality when $\lvert f_z \rvert$ is a constant on $\partial B_1$, forcing either $b_0 = 0$ or $b_1 = 0$ in \eqref{lastbitofit}. We conclude that either $f = C_1 + C_2 z$ or $f = C_1 + C_2/z$ for $C_2 \neq 0$. Hence, $\Omega$ can be obtained from $B_1 - \overline{B}_{\beta}$ by a composition of translation and homotheties.
\end{proof}
The second part of \ref{positivemasstheoremSHEN} in \autoref{Shenwangresult} can be proven as an
 easy corollary of \autoref{keyc3inequalityannulus}. To see that, by \autoref{Mainresultannulus1} and the standard comparison theorem, $\lambda(\Omega,v)= 0$ implies that $\Omega$ must be a simply connected domain. Let $f \colon B_1 \to \Omega$ be a biholomorphic map and define $A(t), B(t)$ accordingly. Then \autoref{keyc3inequalityannulus} implies that $f$ is a M\"obius transformation, and therefore $f(B_1) = \Omega$ must be a disk.
\section{Appendix}
By \cite{MR4308060}, a solution to \eqref{Liouvilleinftyboundaryproblem} on $ B_r(p)-\{p\}$ is given by 
\begin{equation}
u_{B_r(p)-\{p\}}(x) = - \log (-\frac{\abs{x-p}}{r}\cdot \log\bigl(\frac{\abs{x-p}}{r}\bigr))-2 \ln r.
\end{equation}
A solution to $B_1(p)-\overline{B}_r(p)$ is given by 
\begin{equation}
u_{B_1(p)-\overline{B}_r(p)}(x)= - \log \biggl( \frac{\abs{x-p}}{\pi} \cdot \log (1/r)\cdot \sin \bigl(\frac{
\pi\log(1/\abs{x-p})}{ \log (1/r)} \bigr)\biggr).
\end{equation}
Then the growth condition \eqref{growthcondition} is equivalent to $u \geq u_{B_{r_l}(p_l)-\{p_l\}}$  for all $l=1,\cdots, n$.
\begin{theorem}\label{punctureexistence}
Let $\Omega$ be a bounded domain with smooth boundary. Let $p_1, \cdots, p_n \in \Omega$ be a collection of points. Then there exists $w \in C^\infty(M)$ such that 
\begin{equation}
	\begin{cases}
\Delta w = e^{2w} & \text{ in }\Omega - \{p_1, \cdots, p_n\}\\
w = \infty & \text{ on } \partial \Omega \cup \{p_1, \cdots, p_n\},
	\end{cases}
\end{equation}
and \eqref{growthcondition} is satisfied at $\{p_1,\cdots, p_n\}$.
\end{theorem}
\begin{proof}
	We follow the same idea as in \cite{MR1237124}.
Let $u_{\Omega}$ be the solution to \eqref{Liouvilleinftyboundaryproblem} on $\Omega$. Let $A_m$ be an exhaustion of $\Omega$ by compact subsets with smooth boundary. Take $\Omega_m = A_m - \bigcup_{l=1}^n B_{1/m} (p_l)$. Let $w_m$ be the solution to \eqref{Liouvilleinftyboundaryproblem} on $\Omega_m$. By the Maximum principle, $w_m$ is a decreasing sequence. Additionally, $w_m$ is bounded below by $u_{B_{r_l}(p_l)-\{p_l\}}$ for some $r$ sufficiently large, with $l = 1, 2, \cdots, n$. \par  
Fix $m$, by the interior $L^p$ estimate \cite{MR1814364}*{Theorem 9.11}, for $k > m$
\begin{equation}
	\norm{w_k}_{2,p;\Omega_m} \leq C(p,\Omega_m,\Omega_{m +1}) \bigl( \norm{w_k}_{p;\Omega_{m+1}}+ \norm{e^{w_k}}_{p;\Omega_{m+1}}\bigr)
\end{equation}
where  $ \norm{w_k}_{p;\Omega_{m+1}}, \norm{e^{w_k}}_{p;\Omega_{m+1}}$ is bounded independent of $k$, since 
\begin{equation}
\max(u_{B_{r_l}(p_l)-\{p_l\}},u_{\Omega}) \leq w_k \leq w_{m+1}
\end{equation}
for all $l$. This shows that $w_k$ is a bounded sequence in $W^{2,p}(\Omega_m)$. Hence is bounded in $C^{1,\alpha}(\Omega_m)$. \par 
By the interior Schauder estimate, \cite{MR1814364}*{Corollary 6.3}, we have 
\begin{equation}
	\abs{w_{k}}_{2,\alpha;\Omega_m} \leq C(\alpha, \Omega_m,\Omega_{m+1}) \bigl( \abs{w_k}_{0;\Omega_{m+1}} + \abs{e^{2w_{k}}}_{0,\alpha; \Omega_{m+1}}\bigr).
\end{equation}
Hence, $w_k$ is bounded in $C^{2,\alpha}(\Omega_m)$ independently of $k$. Taking a subsequence if necessary, we may assume that $w_m$ converges in $C^2(\Omega_m)$ to some $w$ such that 
\begin{equation}
	\Delta w = e^{2w}
\end{equation}
in $\Omega_m$ and $ \max(u_{\Omega}, u_{B_{r_l}(p_l)-\{p_l\}}) \leq w \leq w_{m+1}$. \par 
Pushing $m$ to $\infty$, $w_m$ converges to a smooth function $w$ satisfying $\Delta w = e^{2w}$ and 
\begin{equation}
 \max (u_{\Omega}, u_{B_{r_l}(p_l)-\{p_l\}})<w.
\end{equation}
 Since $u_{\Omega} = \infty$ on $\partial \Omega$, and $u_{B_{r_l}(p_l)-\{p_l\}} = \infty$ on $p_l$ for all $l$, we see that $w$ is the desired function.
\end{proof}
Following the work of \cite{MR1237124}, we shall establish the uniqueness of solution $u$ to \eqref{Liouvilleinftyboundaryproblem} on $\Omega - \{p_1,\cdots, p_n\}$ satisfying \eqref{growthcondition} in two steps. First we study the asymptotic behavior of a solution near $\{p_1, \cdots, p_n\}$. Then we apply maximum principle to the difference of two solutions. \par 
In the following, we assume without loss of generality that
\begin{equation}
B_2(p_l) \subset \Omega - \{p_1, \cdots, p_n\}
\end{equation}
 for all $l = 1, \cdots, n$.
\begin{proposition}\label{asympointboundary}
Let $u$ be a solution to \eqref{Liouvilleinftyboundaryproblem} on $\Omega - \{p_1,\cdots, p_n\}$. If $u$ satisfies \eqref{growthcondition} at $\{p_1,\cdots, p_n\}$, then there exists a constant $C$ such that 
\begin{equation}
\abs{u - \log \big(-\abs{x-p_l} \cdot \log (\abs{x-p_l}) \bigr)} < C
\end{equation}
in a neighborhood of $p_l$ for all $l=1,\cdots, n$.
\end{proposition}
\begin{proof}
Fix any $l$, and assume, without loss of generality, that $p_l = 0$. By the maximum principle, 
\begin{equation}
 u_{B_{r_l}-\{0\}} \leq u \leq u_{B_1-\overline{B}_r}
\end{equation}
on $B_1 - \overline{B}_r$, for any $0< r < 1$. For $x\in B_1 - \overline{B}_r$, let $\abs{x}$ be the distance to $p_l$. Taking $r = \abs{x}^k$, then 
\begin{equation}
u_{B_1- \overline{B}_r}(x) = -\log \biggl( \frac{-k\abs{x}}{\pi} \cdot \log(\abs{x}) \cdot \sin \bigl( \frac{\pi}{k}\bigr)\biggr).
\end{equation}
Pushing $k \to \infty$, we see that 
\begin{equation}
u \leq - \log \big(-\abs{x} \cdot \log (\abs{x}) \bigr)
\end{equation}
near $p_l$. Now, the difference between the upper bound and the lower bound of $u$ equals to 
\begin{align}
-\log(-\abs{x} \log(\abs{x})) + \log (-r\abs{x} \log(\abs{x}/r))= \log \biggl(\frac{r\log(\abs{x}/r)}{\log (\abs{x})} \biggr),
\end{align}
which converges to $0$ as $\abs{x} \to 0$. Proving the claim.
\end{proof}
\begin{proof}[Proof of \autoref{existenceextension}]
The existence of a solution follows from \autoref{punctureexistence}.
	We follow the same idea as in \cite{MR1237124}.
Let $u_1$ and $u_2$ be two solutions satisfying \eqref{growthcondition} at $p_1, \cdots, p_n$.
By adding a constant if necessary, we may assume that $u_1,u_2 \geq 2 \log 2$; then, for these two functions after modification, there exists some constant $C$ such that 
\begin{equation}
\Delta u_i = C e^{2 u_i}
\end{equation}
on $\Omega$ for $i= 1,2$. \par 
For any $1/2 < s < 1$, by \autoref{asympointboundary}, 
\begin{equation}
\frac{su_1}{u_2} = \frac{s u_1/(-\log(-\abs{x-p_l} \cdot \log(\abs{x - p_l})))}{ u_2/(-\log(-\abs{x-p_l} \cdot \log(\abs{x - p_l})))} \to s < 1
\end{equation}
near $p_l$ for all $l = 1,2, \cdots, n$. By \cite{MR3758549}, near other non-degenerate boundary components, $su_1/u_2$ also converges to $s < 1$. \par 
We claim that $su_1 \leq u_2$. If not, let $x_0$ be the point where $su_1 - u_2$ attains its maximum. Then $(su_1 - u_2)(x_0) > 0$, and 
\begin{equation}
0 \geq \Delta (su_1 - u_2)(x_0)= s Ce^{2u_1} - C e^{2 u_2}\geq e^{2s u_1} - e^{2 u_2} > 0,
\end{equation}
where the second inequality comes from \cite{MR1237124}*{Lemma 4.3}, which leads to a contradiction. \par 
Pushing $s \to 1$, we get $u_1 \leq u_2$. By switching $u_1$ and $u_2$, we get the inequality from other direction, thereby proving the claim. 
\end{proof}
\bibliography{A_new_renormalized_volume_type_invariant_submit}
\end{document}